\documentclass[final,1p,times]{elsarticle}
\usepackage{color}
\usepackage{ulem}
\usepackage{verbatim}
\usepackage{latexsym}
\usepackage{amssymb,amsmath,amsthm,mathrsfs}

\usepackage[utf8]{inputenc}
\usepackage{aeguill}
\usepackage[english]{babel}
\usepackage{amsmath}
\usepackage{amsfonts}
\usepackage{amsthm}
\usepackage{graphicx}
\usepackage{subfigure}
\usepackage[nottoc, notlof, notlot]{tocbibind}
\usepackage{fmtcount}

\usepackage{breqn}
\usepackage[lined,linesnumbered]{algorithm2e}

\theoremstyle{plain}%
\newtheorem{thm}{Theorem}[section]
\newtheorem{lemma}[thm]{Lemma}
\newtheorem{prop}[thm]{Proposition}
\newtheorem{cor}{Corollary}
\newtheorem{claim}{Claim}
\theoremstyle{definition}
\newtheorem{defn}{Definition}%

\newtheorem{exmp}{Example}[section]
\theoremstyle{remark}
\newtheorem{remark}{Remark}

\newcommand{\F}{\mathbb{F}_2}
\newcommand{\Fn}{\mathbb{F}_{2^n}}

\newcommand{\de}[1][d]{{\normalfont\texttt{Inv}_{{#1}}}\,}
\newcommand{\dego}[1][2^k+1]{{\normalfont\texttt{Inv}_{{#1}}}\,}
\newcommand{\dekas}[1][2^{2k}-2^k+1]{{\normalfont\texttt{Inv}_{{#1}}}\,}

\newcommand{\ordb}[1][d]{{{\theta}_{#1}\,}}

\newcommand{\spdb}{{\normalfont\texttt{S}_{d}}\,}

\begin{document}

\begin{frontmatter}
\author{Gohar M. Kyureghyan}
\address{Department of Mathematics, Otto-von-Guericke University,\\ Universit\"atsplatz 2, 39106 Magdeburg, Germany \\
Email: gohar.kyureghyan@ovgu.de}

\author{Valentin Suder}
\address{INRIA Paris-Rocquencourt, Project-Team SECRET\\ Domaine de Voluceau, 78153 Le Chesnay, France \\ Email: valentin.suder@inria.fr}

\title{On Inversion in $\mathbb{Z}_{2^n-1}$}
\date{}

\begin{abstract}
In this paper  we determined explicitly the multiplicative inverses of the Dobbertin and Welch APN exponents in $\mathbb{Z}_{2^n-1}$,
and we described the binary weights of the inverses of the Gold and Kasami exponents. 
We studied the function $\de(n)$, which for a fixed positive integer $d$ maps integers $n\geq 1$ to
the least positive residue of the inverse of $d$ modulo $2^n-1$, if it exists. In particular, we showed that the function $\de$ is completely
determined by its values for  $1 \leq n \leq \ordb$, where $\ordb$ is the order of $2$ modulo the largest odd divisor of $d$.\footnote{The first part of this work is an extended version of the results presented in ISIT12 \cite{KS12}.}
\end{abstract}


\end{frontmatter}

\section{Introduction}

A mapping $f :\Fn \to \Fn$ is called {\it almost perfect nonlinear} (APN) if for every non-zero
$a \in \Fn$ the sets 
$$
\{ f(x+a) +f(x) ~:~x \in \Fn\}
$$
contain exactly $2^{n-1}$ elements. When $n$ is odd, a mapping $f$ is called {\it almost bent} (AB) if for every $\alpha \ne 0,~ \beta \in \Fn$ 
$$ \sum_{x\in \Fn} (-1)^{Tr\left(\alpha F(x) + \beta x\right) } \in \{0,~ \pm 2^{\frac{n+1}{2} }\}, $$
where $Tr$ is the absolute trace on $\Fn$. Every AB mapping is APN, but not vice versa.
APN and AB mappings have various applications in cryptology, coding theory and combinatorics \cite{CCD00,CCZ98,HouMullen,N93}.

Every mapping $f$ of $\Fn$ has a unique univariate polynomial representation over $\Fn$ of degree
not exceeding $2^n-1$. With respect to a fixed $\F$-basis of $\Fn$, the mapping $f$ has a unique multivariate representation over $\Fn$
such that its degree in every single variable is less than 2. This multivariate polynomial representation is basis
dependent. However its total degree does not depend on the basis choice, and it is called the algebraic degree of the mapping $f$.
The algebraic degree of a mapping can be computed from its univariate polynomial representation:
Recall that the  binary weight of a nonnegative integer $d$ is the sum of the digits in its binary representation, i.e. if
$d= \sum_{i=0}^l d_i2^i$ with $0\leq d_i \leq 1$, then the binary weight of $d$ is $\mbox{wt}(d)=\sum_{i=0}^ld_i \in \mathbb{Z}$. 
The algebraic degree of the mapping 
$f(x) = \sum_{k=0}^{2^n-1}\alpha_k x^k$ on $\Fn$ is equal to $\max_{k, \alpha_k \ne 0}\{\mbox{wt}(k)\}$. 
In particular, a monomial mapping given by $x \mapsto x^d$ with $1 \leq x \leq 2^n-2$ has algebraic degree 
equal to $\mbox{wt}(d)$.

When studying a special class of mappings of finite fields, one of the
main questions to be answered is: What are the properties of polynomials describing this  class of mappings? 
This question is widely open for AB/APN mappings. Even a much weaker question, what are the possible degrees for 
univariate or multivariate representations of APN mappings, is one of the open challenges in this research area.
It is known that the algebraic degree of AB mappings does not exceed $\frac{n+1}{2}$ \cite{CCZ98}.

The two best understood classes of APN mappings are the so-called quadratic and monomial ones. The univariate representation
of a quadratic mapping contains only terms with exponents of binary weight less or equal to 2, i.e. it is of the shape $\sum_{i,j} b_{i,j}x^{2^i+2^j}$ $ \in \Fn[x]$. The monomial mappings are those
of shape $x \mapsto x^d$ with a fixed integer  $1\leq d \leq 2^n-2 $. An integer $1 \leq d \leq 2^n-2$ is called APN exponent on $\Fn$ if the corresponding monomial 
mapping $x \mapsto x^d$ is APN on $\Fn$. All currently known APN exponents can be obtained from the ones listed below:

\begin{table}[htb]
\begin{center}
\begin{tabular}{|c|c|c|c|}
\hline
& Exponents~\(d\)&Conditions &\\
\hline
\hline
Gold  & \(2^k+1\) & \(\gcd(k,n)=1\), & APN\\
& &\(1 \leq k \leq t\) &  AB if $n$ is odd\\
\hline
Kasami  & \(2^{2k} - 2^k +1\) & \(\gcd(k,n)=1\) & APN\\
& & \(2 \leq k \leq t\) & AB if $n$ is odd\\
\hline
Welch  & \(2^t+3\) & $n=2t+1$ & APN/AB\\
\hline
Niho  & \(2^t + 2^{\frac{t}{2}}-1\) if \(t\) is even & $n=2t+1$ & APN/AB\\
& \(2^t + 2^{\frac{3t+1}{2}}-1\) if \(t\) is odd & & \\
\hline
\hline
inverse & \(2^{2t}-1\) & $n=2t+1$ & APN  \\
\hline
Dobbertin & \(2^{4t} + 2^{3t} + 2^{2t} + 2^t -1\) & \(n =5t\) & APN\\
\hline
\end{tabular} 
\caption{Exponents defining APN/AB monomial mappings on $\Fn$}
\label{tab:APN}
\end{center} 
\end{table} 

It is easy to prove that if $d$ is an APN exponent then also $2\cdot d \pmod {2^n-1}$ is so, as well as the inverse $d^{-1}$ of $d$ modulo $2^n-1$
if it exists. While the multiplication of an APN exponent $d$ by $2$ is a fairly easy operation, a better understanding of the inverse of $d$ will yield more insights
on APN mappings. It is well known that an APN exponent on $\Fn$ is invertible in $\mathbb{Z}_{2^n-1}$ if and only if $n$ is odd. In \cite{N93, PR97} the inverses of  Gold's and Niho's exponents were considered. In this paper we continue this study.
In particular, we find explicit formulas for the inverses of Welch's and Dobbertin's exponents (see Theorem \ref{thm:welch} and \ref{th:dobbertin} respectively) and obtain some partial
results on the inverses of  Gold's and Kasami's exponents (see Theorem \ref{main:gold} and \ref{thm:kasami} respectively for the main results). Further we study also the inverses of some other interesting classes
of exponents. 

When studying  inverses of APN exponents $d$ on $\Fn$, two cases must be distinguished:
\begin{itemize}
\item $d$ depends on $n$ (Dobbertin's, Niho's, Welch's exponents and the field inverse)
\item a fixed $d$ is an APN exponent on $\Fn$ for infinitely many $n$ (Gold's and Kasami's exponents).
\end{itemize}

It appears that the study of the latter exponents is more difficult than the study of the exponents of the first type. 
The  exponents $d$ defining APN mappings on $\Fn$ for infinitely many $n$ are called {\it exceptional APN exponents} \cite{dillon}.
In \cite{HMc11}, it is shown that Gold's and Kasami's exponents are the only exceptional APN ones. 
Finding the inverses of exceptional APN exponents is an instance of
the following general problem:  Let $d \geq 1$ be a fixed integer and define $\mathbb{N}_d =\{ n \in \mathbb{N} ~:~ \gcd(d,2^n-1) = 1\}$.
How explicit can we describe the function $\de :\mathbb{N}_d \to \mathbb{N}$, which maps $n$ to the least positive
integer describing the inverse of $d$ modulo $2^n-1$? We show that the function $\de$ is completely
determined by its values for  $1 \leq n \leq \ordb'$, where $\ordb'$ is the order of $2$ modulo the largest odd divisor of $d$.
This dependence is given in Theorem \ref{main}. Based on properties of $\de$ we propose Algorithm \ref{algo} for inversion
in $\mathbb{Z}_{2^n-1}$, which may be of interest for some special applications.

\section{Inverses of APN exponents $d$ on $\Fn$, when $d$ depends on $n$}\label{nkdepend}

 After having the conjectural modular inverse for a given integer, usually the correctness of it follows 
from easy  calculations. Hence the main difficulty in finding  inverses   is to guess them. 
For a generic integer $d$  we cannot of course expect to be able to guess its inverse modulo $2^n-1$.
However, if the binary representation of the integer $d$ has a nice combinatorial
pattern, then   the binary representation of its modulo $2^n-1$ inverse does not look
random either, and therefore the problem could be solvable.  Since the binary representations of the APN exponents listed in Table \ref{tab:APN} are far from being random,
finding their inverses explicitly is probably possible. Having this idea in mind we performed
numerical experiments, which led to the formulas for Dobbertin's and Welch's exponents described 
 below.

\subsection{Dobbertin's exponent}

\begin{thm}\label{th:dobbertin}
Let $k\geq 1$ be an odd integer. The least positive residue of the inverse of $d = 2^{4k} + 2^{3k} + 2^{2k} + 2^k - 1$ modulo $2^{5k}-1$ is 
$$
t =  \frac{1}{2} \left( \frac{2^{5k}-1}{2^k -1 } \cdot \frac{2^{k+1}-1}{3} -1\right).
$$
Furthermore,
$$
2\cdot t = \left(\sum_{i=0}^{4} \sum_{j=0}^{\frac{k-1}{2}} 2^{ik + 2j - 1}\right) - 1 ,
$$
showing that
$wt(t) = \frac{5k + 3}{2}.$
\end{thm}

\begin{proof} Set $A := \frac{2^{5k}-1}{2^k -1 } \cdot \frac{2^{k+1}-1}{3}$.
Then 
\begin{displaymath}
2^{k} \cdot A \equiv A \pmod{2^{5k} - 1}.
\end{displaymath}
Observe that $d = \frac{2^{5k}-1}{2^k-1} -2$. Hence
\begin{eqnarray*}
d\cdot t & \equiv &  \frac{1}{2} ( 3A - d) = \frac{1}{2} \left( \frac{2^{5k}-1}{2^k -1 } \cdot (2^{k+1}-1) - \frac{2^{5k}-1}{2^k-1} + 2\right) \\
         & = & \frac{1}{2} \left( \frac{2^{5k}-1}{2^k -1 } \cdot (2^{k+1}-2) +2 \right) \\
         & \equiv & 1 \pmod{2^{5k}-1}.
\end{eqnarray*}
Clearly, $t < 2^{5k}-1$ and thus $t$ is indeed the least positive residue of the inverse of $d$ modulo $2^{5k}-1$.
\end{proof}

 By Theorem \ref{th:dobbertin} the inverse of Dobbertin exponent defines an APN mapping on $\Fn$
with algebraic degree exceeding $(n+1)/2$. The only previously known such example was the inverse
mapping, with algebraic degree $n-1$. This observation  shows also that
the monomial mappings defined by  Dobbertin's exponents and their inverses are not AB,
which was originally proved in \cite{CCD00} by exploiting the divisibility properties of  corresponding codes.

\begin{cor}
The monomial mappings with Dobbertin's exponents and their inverses are not $AB$.
\end{cor}

\subsection{Niho's exponents}

For the sake of the completeness, we give here the explicit inverses of Niho's exponents which were found in \cite{PR97}.
For $n = 2m + 1$, Niho's exponent $d$ has the shape
$$
d = \left\{ \begin{array}{ll}  
2^m + 2^{\frac{m}{2}}-1 & \textrm{ if } m \textrm{ is even} \\
        2^m + 2^{\frac{3m+1}{2}}-1 & \textrm{ if } m \textrm{ is odd},
\end{array} 
\right .
$$
or equivalently, with $k \geq 1$,
$$
d = \left\{ \begin{array}{ll}  
2^{2k} + 2^{k}-1 & \textrm{ if } m =2k \\
        2^{3k+2} + 2^{2k+1}-1 & \textrm{ if } m=2k+1.
\end{array} 
\right .
$$
The inverses for Niho exponents depend on $n \pmod 8$ as the next theorem shows.

\begin{thm}\label{thm:niho1}
{\bf(a)} Let $n = 4k + 1$ and $d = 2^{2k} + 2^k - 1$ be the Niho exponent. Set $t$ to denote the least positive residue  of the inverse of $d$ 
modulo $2^n-1$. Then 
\begin{itemize}
\item if $k$ is even, i.e. if $n \equiv 1 \pmod 8$,
\begin{equation}\label{eq:niho1}
t = \frac{2^k-1}{3}\left(2^{3k+1}+2^{k+1} +1\right) + 2^k + 2^{3k+1}.
\end{equation}

\item if $k$ is odd, i.e. if $n \equiv 5 \pmod 8$,
\begin{equation}\label{eq:niho2}
t = \frac{2^{k-1} -1}{3}\left(2^{3k+2} + 2^{2k+2} + 1\right) + 2^{3k+1} + 2^{2k+1} + 2^{k-1}. 
\end{equation}

\end{itemize}
In particular,
$$
wt(t) = \left\{ \begin{array}{ll} \vspace*{0.2cm}
\frac{3n + 5}{8} & \textrm{ if } n \equiv 1  \pmod 8 \\ 
\frac{3n + 9}{8} &  \textrm{ if } n \equiv 5 \pmod 8.
\end{array}
\right .
$$

{\bf(b)} Let $n = 4k+3$ and $d = 2^{3k+2} + 2^{2+1} - 1$ be the Niho exponent.
Set $t$ to denote the least positive residue  of the inverse of $d$ 
modulo $2^n-1$. Then
\begin{itemize}
\item if $k$ is even, i.e. if $n \equiv 3 \pmod 8$,
\begin{equation}\label{eq:niho3}
 t = \frac{2^k-1}{3} \left( 2^{3k+4} + 2^{k+2} +2 \right) + 2^{3k+3} + 2^{k+1}.
\end{equation}
\item if $k$ is odd, i.e. if $n \equiv 7 \pmod 8$,
\begin{equation}\label{eq:niho4}
 t = \frac{2^{k+1} - 1}{3} \left( 2^{3k+3} + 2^{2k+3} +2 \right) + 2^{2k+2}.
\end{equation}
\end{itemize}
  In particular, 
$$
  wt(t) = \left\{ \begin{array}{ll} \vspace*{0.2cm}
  \frac{3n+7}{8} & \textrm{ if }  n \equiv 3 \pmod 8 \\
  \frac{3n+11}{8} & \textrm{ if } n \equiv 7 \pmod 8.
\end{array} \right .
$$
\end{thm}

\begin{proof}
We prove only the first statement of part {\bf(a)}, since the remaining cases follow by similar arguments.
\begin{itemize}
\item Let $n = 4k+1$ and $k$ be even. Then
\begin{eqnarray*}
 d \cdot t & = &(2^{2k} + 2^k -1) \left( \frac{2^k -1}{3} \left(2^{3k+1} + 2^{k+1} +1\right) +2^k + 2^{3k+1} \right) \\
 & \equiv & \frac{2^k -1}{3} \left(2^{5k+1} + 2^{4k+1} - 2^{3k+1} + 2^{3k+1} + 2^{2k+1} -2^{k+1} + 2^{2k} + 2^k -1 \right)\\
 & & + 2^{3k} + 2^{2k} - 2^k + 2^{5k+1} + 2^{4k+1} - 2^{3k+1}\\
 & \equiv & \frac{2^k -1}{3} \left(2^{2k+1} + 2^{2k} \right) -2^{3k} + 2^{2k} + 1 \\
&\equiv& (2^k-1)2^{2k} -2^{3k} + 2^{2k} + 1  \\
&\equiv & 1 \pmod{2^{4k+1}-1}.
\end{eqnarray*}
It remains to note that $\mbox{wt}(t) = 2 + 3 \cdot \mbox{wt}\left(\frac{2^k -1}{3}\right) -1 = 1 + 3 \cdot \frac{n-1}{8} = \frac{3n +5}{8}$, since 
$$
t = \frac{2^k -1}{3} + 2^k + 2^{k+1}\cdot \frac{2^k -1}{3}  + 2^{3k + 2} + 2^{3k+1} \cdot \left(\frac{2^k -1}{3}-1  \right)
$$
and
$$\frac{2^k-1}{3} = \sum_{i=0}^{\frac{k}{2}-1}2^{2i}.$$ 

\end{itemize}
\end{proof}

\subsection{Welch's exponent}

Let $v$ be a nonnegative integer with binary representation $v = \sum_{i=0}^{r-1} v_i 2^i$ where $v_{r-1} \ne 0$. If $r' \geq r$, we say that 
$$
\underbrace{0\,0\ldots 0}_{r'-r}\,v_{r-1}\, \ldots v_1\,v_0
$$
is the sequence of length $r'$ representing the integer $v$. Two sequences $a_{r-1} \ldots a_0$ and $b_{r-1} \ldots b_0$ are called complementary if $a_i = b_i +1 \pmod 2$ for all $0\leq i \leq r-1$.
For a sequence  ${\boldsymbol a}$, we denote  by  ${\boldsymbol{\bar{a}}}$ its complementary sequence. Note that if a sequence $\boldsymbol{a} = a_{r-1} \ldots a_0$
represents the integer $a$, then its complementary sequence $\boldsymbol{\bar{a}}$ represents the integer $2^r-1-a$. This follows from the fact that 
the sequence $\underbrace{1\,1 \ldots 1}_{r}$ represents the integer $2^r-1$. For two sequences ${\boldsymbol a}$ and ${\boldsymbol b}$, we denote 
by $\boldsymbol{a | b}$ their concatenation.

\begin{lemma}\label{lem:weight} Let $s \geq 2$ and $0< u < 2^s$ be integers. Then the binary representation of length $2s$ of the
integer $(2^s-1)\cdot u$ is ${\boldsymbol{w | \bar w}}$, where $\boldsymbol{w}$ is the sequence of length $s$ representing the integer $u-1$
and ${\boldsymbol{ \bar w}}$ is the complement of the binary sequence $\boldsymbol{w}$ and represents the integer $2^s-u$.
In particular, the binary weight of the integer $(2^s-1)\cdot u$ is $s$.
\end{lemma}
\begin{proof}
The statement of the lemma follows from the observation that $(2^s-1)\cdot u = (u-1)\cdot 2^s + (2^s-u)$. Moreover, the binary representation
of the length $s$ of $2^s-u$ is the complement of that of $u-1$, since $(2^s-u) + (u-1) = 2^s-1$.
\end{proof}

\begin{thm}\label{thm:welch}
Let $n = 2k+1$. The least positive residue $t$ of  the inverse of Welch's exponent $2^k+3$ modulo $2^{2k+1}-1$ is: 
\begin{itemize}
  \item If $k \equiv 0 \pmod 8$ then \\
$$t = 2^k + \frac{2^{k}-1}{17}\left(13 \cdot 2^{k+1} + 7 \right) $$ with binary weight $k+1$.
  \item If $k \equiv 1 \pmod 8$ then \\
$$t = 2^{k-1} +2^{k} + \frac{2^{k-1}-1}{17}\left(7 \cdot 2^{k+2} + 1 \right)$$ with binary weight $k+1$.
  \item If $k \equiv 2 \pmod 8$ then \\
$$t = 1 +2^{k+1}  + \frac{2^{k-2}-1}{17}\left(5 \cdot 2^{k+3} + 16 \right)$$ with binary weight $k$.
  \item If $k \equiv 3 \pmod 8$ then\\
$$t = 2^k +2^{k+2}+2^{k+3}  + \frac{2^{k-3}-1}{17}\left(7 \cdot 2^{k+5} + 8 \right)$$ with binary weight $k$.
  \item If $k \equiv 4 \pmod 8$ then\\
$$ t = 2^{k-4} +2^{k-2}+2^{k-1}+2^{k+4}  + \frac{2^{k-4}-1}{17}\left(9 \cdot 2^{k+5} + 3 \right)$$ with binary weight $k$.
  \item If $k \equiv 5 \pmod 8$ then \\
$$t = 1+ 2^{k-3} +2^{k-1}+2^k +2^{k+1}   + \frac{2^{k-5}-1}{17}\left( 2^{k+6} + 12 \right)$$ with binary weight $k$.
  \item If  $k \equiv 6 \pmod 8$ then \\
\begin{align*}
t &= 2^{k-5} +2^{k-4}+2^{k-2}+ 2^{k+3} +2^{k+4}  \\
& \quad + 2^{k+5} +2^{k+6} + \frac{2^{k-6}-1}{17}\left( 16\cdot 2^{k+7} + 10 \right)
\end{align*}
with binary weight $k+1$.
  \item If $k \equiv 7 \pmod 8$ then \\ 
\begin{align*}
t &= 2^{k-5} +2^{k-4}+ 2^{k-3} +2^{k-2}  + 2^{k+1} +2^{k+2}  \\ 
 & \quad + 2^{k+4} +2^{k+7} + \frac{2^{k-7}-1}{17}\left( 10\cdot 2^{k+8} + 4 \right)
\end{align*}
 with binary weight $k+1$.
\end{itemize}
\end{thm}

\begin{proof}
We only need to verify that $(2^k+3)\cdot t \equiv 1 \pmod {2^{2k+1}-1}$, since obviously all listed integers $t$ are smaller than $2^{2k+1}-1$. 
We do this verification for $k \equiv 4 \pmod 8$, the remaining cases are similar. Thus, let $k \equiv 4 \pmod 8 $ and consider
\begin{equation}\label{congr:welch}
(2^k+3)\cdot \left(2^{k-4} +2^{k-2}+2^{k-1}+2^{k+4}  + \frac{2^{k-4}-1}{17}\left(9 \cdot 2^{k+5} + 3 \right)\right) \pmod{2^{2k+1}-1}.
\end{equation}
Observe that
$$
(2^k+3) \cdot \left(9 \cdot 2^{k+5} + 3 \right) \equiv 17 \cdot (9 + 3\cdot 17 \cdot 2^k) \pmod{2^{2k+1}-1},
$$
and therefore (\ref{congr:welch}) reduces to
\begin{eqnarray*}
(2^k+3)\cdot \left(2^{k-4} +2^{k-2}+2^{k-1}+2^{k+4} \right) + (2^{k-4}-1) \cdot (9 + 3\cdot 17 \cdot 2^k) = \\
(2^k+2 +1) \left(2^{k-4} +2^{k-2}+2^{k-1}+2^{k+4} \right) + (2^{k-4}-1)  (2^3+1 + 2^{k+5}+2^{k+4}+2^{k+1}+2^k) \\ \equiv 
2^{2k-4} +2^{2k-2}+2^{2k-1}+2^{3} +2^{k-3} +2^{k-1}+2^{k}+2^{k+5} +2^{k-4} +2^{k-2}+2^{k-1}+2^{k+4} +\\
2^{k-1}+2^{k-4} + 1 +2^{2k}+2^{2k-3}+2^{2k-4} - 2^3 -1 - 2^{k+5} -2^{k+4} -2^{k+1} - 2^k = \\
2^{2k-4} +2^{2k-2}+2^{2k-1} +2^{k-3} +2^{k-1} -2^{k} +2^{k-4} +2^{k-2} + 2^{k-4}  +2^{2k}+2^{2k-3}+2^{2k-4}      = \\
2^{2k} + 2^{2k-1} + 2^{2k-2}  + 2^{2k-3} +2 \cdot 2^{2k-4} -2^k + 2^{k-1} +2^{k-2} +  2^{k-3} + 2 \cdot 2^{k-4} \equiv \\
1 \pmod{2^{2k+1}-1}. 
\end{eqnarray*}
To complete the proof it remains to show that the binary weight of $t = 2^{k-4} +2^{k-2}+2^{k-1}+2^{k+4}  + \frac{2^{k-4}-1}{17}\left(9 \cdot 2^{k+5} + 3\right) $ is $k$.
We firstly compute the binary weight of
\begin{eqnarray*}
\frac{2^{k-4}-1}{17}\left(9 \cdot 2^{k+5} + 3\right) &=& \frac{2^{k-4}-1}{2^8-1}(2^4-1)\left(9 \cdot 2^{k+5} + 3\right) \\
& = & (2^4-1)\left(9 \cdot 2^{k+5} + 3\right)\sum_{j=0}^{\frac{k-4}{8}-1}2^{8j} \\
& = & 3 (2^4-1) \sum_{j=0}^{\frac{k-4}{8}-1}2^{8j} + 9(2^4-1) \sum_{l=\frac{k+4}{8}}^{\frac{k}{4}-1}2^{8l+1}.
\end{eqnarray*}
Note that the integers $3$ and $9$ are less than $2^4$, so Lemma \ref{lem:weight} implies that the binary weight of both integers
$3 (2^4-1) \sum_{j=0}^{\frac{k-4}{8}-1}2^{8j}$ and   $9(2^4-1) \sum_{l=\frac{k+4}{8}}^{\frac{k}{4}-1}2^{8l+1}$ is $4\cdot \frac{k-4}{8} = \frac{k-4}{2}$. Thus the binary weight of $t$ is
$2\cdot\frac{k-4}{2} +4 = k.$ 
\end{proof} 

\begin{remark}
The crucial step for guessing the inverses $t$ of Welch's exponent was the observation that $t$ satisfies certain recurrence relations.
For instance, we take $n = 2k +1$ with $k \equiv 0 \pmod 8$. Set $k = 8r$ with $r\geq 0$. Suppose $\boldsymbol{{t_r}}$ is the binary sequence of length $n = 16r+1$
representing the least positive residue of the inverse of Welch's exponent $2^{8r}+3$ modulo $2^n-1$. Then for every $r \geq 1$
$$
\boldsymbol{{t_r}} = 1100\, 0011 \, |\, \boldsymbol{{t_{r-1}}}\, |\, 0110 \, 1001
$$
holds and $\boldsymbol{{t_0}} = 1$.
\end{remark}

\section{Inverses of a fixed integer $d$ modulo all $2^n-1$}\label{section:exceptionals}

In the previous section we described   the inverses for  Dobbertin's,
Niho's and Welch's exponents, and herewith 
it remains to find the inverses for Gold's
and Kasami's exponents 
to have all presently known APN exponents explicitly.
A fixed Gold's or Kasami's exponent defines an APN mapping on infinitely
many finite fields and therefore we aim  finding the inverses of a fixed integer modulo infinitely
many $2^n-1$. For example,  Gold's exponents $3$ and Kasami's exponent $13$
define bijective monomial APN mappings on $\Fn$ with any $n$ odd. Hence we want
to find inverses of $3$ and $13$ modulo all  $2^n-1$ where $n$ is odd. 
Motivated by this observation, in the next subsection we study
the following general problem:  For a given fixed integer $d \geq 1$, let $\mathbb{N}_d =\{ n \in \mathbb{N} ~:~ \gcd(d,2^n-1) = 1\}$.
What can we say about the function $\de :\mathbb{N}_d \to \mathbb{N}$, whose output for $n$ is the least positive
integer describing the inverse of $d$ modulo $2^n-1$\,? Most of the results of the next subsection can be
directly generalized for modulo $p^n-1$, where $p$ is a prime number.

\subsection{General case}

\begin{defn}
Let  $d$ be a fixed positive integer. 
For a positive integer $n$ satisfying  $\gcd(d, 2^n-1) =1$, set $\de(n)$ be the least positive residue of the inverse of
$d$ modulo $2^n-1$, that is the integer $\de(n)$ is defined by
$$
0< \de(n) < 2^n-1 ~~ \textrm{ and }~~ d \cdot \de(n) \equiv 1 \!\!\!\pmod{2^n-1}.
$$
\end{defn}

In the rest of this section we assume, without loss of generality,  that the fixed number $d$ is {\bf odd}.
Indeed, if $d_1 = 2^u \cdot d$ with $ u \geq 0$  then the study of the function
$\de[d_1]$
may be reduced to the one of $\de$ using 
$$
\de[d_1](n) \equiv  2^{n-u}\de(n) \pmod{2^n-1}.
$$

Let $d \geq 3$ and  $\ordb$  be the (multiplicative) order of $2$ modulo $d$, that is $\ordb$ is the least positive integer $o$
such that $2^{o}  \equiv 1 \pmod{d}$. 
We set 
$\ordb[1]\,=0$.
The next results show that in the study of the function $\de$ the magnitude $\ordb$  plays an important role:
 $\lfloor \ordb /2\rfloor$ many values of the function $\de$ completely determine it.

The following proposition shows that if $\de(r)$ is known for some $1 \leq r \leq \ordb-1$, then it yields
the value of $\de(\ordb -r)$.
\begin{prop} \label{prop:inverse-r}
Let $1 \leq r \leq \ordb -1$ and $\gcd(d, 2^r-1) =1$. Then $\gcd(d, 2^{\ordb -r}-1) =1$ and
\begin{equation}\label{inverse-r}
\de(\ordb -r) = \frac{\left (d+1 - \frac{\de(r) \cdot d -1}{2^r -1}\right)(2^{\ordb -r} -1) +1 }{d}.
\end{equation}

\end{prop}

\begin{proof}
Note that $\gcd(d, 2^{\ordb -r}-1) =1$, since $2^r\cdot (2^{\ordb -r}-1) = 2^{\ordb} -1 -( 2^r -1)$.
Set $t$ be the rational number at the right hand side of (\ref{inverse-r}), i.e.
$$
t :=  \frac{\left (d+1 - \frac{\de(r) \cdot d -1}{2^r -1}\right)(2^{\ordb -r} -1) +1 }{d}.
$$
First we  show that  $t$ is an integer, or equivalently that
$$
\left (d+1 - \frac{\de(r) \cdot d -1}{2^r -1}\right)(2^{\ordb -r} -1) +1 
$$
is divisible by $d$. Since $\gcd(d, 2^r-1)=1$, it is enough to show that
$$
(2^r-1)\left(\left (1 - \frac{\de(r) \cdot d -1}{2^r -1}\right)(2^{\ordb -r} -1) +1\right) \equiv 0 \pmod d,
$$
which in its turn reduces to
$$
(2^r-1)2^{\ordb -r}  - (\de(r) \cdot d -1)(2^{\ordb -r} -1) \equiv 0 \pmod d. 
$$ 
The left hand side of the latter congruence is
$$
2^{\ordb} -1 - \de(r) \cdot d \cdot \left(2^{\ordb -r}-1\right),
$$
which is divisible by $d$ by the definition of $\ordb$.

Finally it is easy to see that $t\cdot d \equiv 1 \pmod {2^{\ordb -r} -1}$ and $1 \leq t \leq 2^{\ordb -r} -1$.

\end{proof}

The following identity is obtained directly from (\ref{inverse-r}):

\begin{cor}\label{cor:sum}
Let $1 \leq r \leq \ordb -1$ with $\gcd(d, 2^r-1) =1$. Then 
\begin{equation}\label{sum-r}
\frac{d\cdot \de(\ordb -r) -1}{2^{\ordb -r} -1} + \frac{\de(r) \cdot d -1}{2^r -1} = d+1.
\end{equation}
\end{cor}

The next theorem is the main result of this section. It shows that the value of $\de(n)$
can be computed from $\de(r)$ where $r$ is the least positive residue of $n$ modulo $\ordb$.
We give three different expressions for this dependence, each of them appears to be
more convenient for a certain situation.  
Observe that  Theorem \ref{main} implies in particular that in order to determine values of the inverses of $d$ modulo all $2^n-1$ it is enough to compute
only finitely  many of them. 

\begin{thm}\label{main}
Let $n\geq 1$ with $\gcd(d, 2^n-1) =1$ and
$1 \leq r \leq \ordb -1$ such that $n \equiv r \pmod \ordb$. Then
\begin{description}
\item[(a)]
\begin{equation}\label{eq:main1}
\de(n) = \de(r) \cdot 2^{n-r} + \left( \frac{d \cdot \de(r)-1}{2^r-1} - 1 \right) \cdot\frac{2^{n-r}-1}{d}.
\end{equation}
\item[(b)]
\begin{equation}\label{eq:main2}
\de(n) = \de(r) \cdot \sum_{i=0}^{m} 2^{\ordb\cdot i} + \left( 2^{\ordb-r}-1 - \de(\ordb -r) \right) 2^r \sum_{i=0}^{m-1} 2^{\ordb \cdot i},
\end{equation}
where $m = \frac{n-r}{\ordb}.$ Equivalently,
\begin{eqnarray*}
\de(n) & = & \de(r) \cdot \frac{2^{\ordb\cdot(m+1)} -1}{2^{\ordb}-1} + \left( 2^{\ordb-r}-1 - \de(\ordb-r) \right)
 2^r \cdot \frac{2^{\ordb\cdot m} -1}{2^{\ordb} -1}.
\end{eqnarray*}
\item[(c)]
\begin{eqnarray*}
\de(n) & = & \frac{\de(r)(2^n-1) - \frac{2^n-2^r}{d}}{2^r-1}.
\end{eqnarray*}
\end{description}
\end{thm}

\begin{proof}
{\bf (a)} Set 
$$
t:=\de(r) \cdot 2^{n-r} + \left( \frac{d \cdot \de(r)-1}{2^r-1} - 1 \right) \cdot\frac{2^{n-r}-1}{d}.
$$
Clearly, $t$ is a positive integer, since $2^r-1$ divides $d\cdot \de(r)-1$ and $d$ divides $2^{n-r}-1$.
Next we show that $t < 2^n-1$. Note that $\de(r) < 2^r-2$, since otherwise $\de(r) = 2^r-2 =d$, which
contradicts the assumption that $d$ is odd.
Hence we have
\begin{eqnarray*}
\de(r) &<& 2^r-2  ~~\Rightarrow \\ d \cdot \de(r) &<& d \cdot (2^r-2)  ~~\Rightarrow \\
d \cdot \de(r) -1 & <& d \cdot (2^r-1) - (d+1) ~~ \Rightarrow \\ 
 \frac{(d  \cdot \de(r) -1)(2^n-1)}{2^r-1} & <& (2^n-1) \cdot d - \frac{(d+1)(2^n-1)}{2^r-1}  ~~ \Rightarrow\\
 \frac{(d  \cdot \de(r) -1)(2^n-1)}{2^r-1} +1 & <& (2^n-1) \cdot d - (d+1) +1  ~~ \Rightarrow \\
 \frac{(d  \cdot \de(r) -1)(2^n-1)}{2^r-1} +1 & <& (2^n-1) \cdot d  ~~ \Rightarrow\\
(2^n-2^r) \cdot \frac{d  \cdot \de(r) -1}{2^r-1} + (d  \cdot \de(r) -1) +1 &  <& (2^n-1) \cdot d.  \\
\end{eqnarray*}
It remains to observe that the left hand side of the last inequality is
\begin{eqnarray*}
(2^n-2^r) \cdot \frac{d  \cdot \de(r) -1}{2^r-1} + (d  \cdot \de(r) -1) +1 =\\ 
2^{n-r} \left(d  \cdot \de(r) -1\right) +  \left(2^{n-r}-1\right)\frac{d  \cdot \de(r) -1}{2^r-1} +1 = \\
2^{n-r} \cdot d \cdot \de(r) +  \left(2^{n-r}-1\right)\left(\frac{d  \cdot \de(r) -1}{2^r-1} -1\right) = t \cdot d,
\end{eqnarray*}
proving that indeed $t < 2^n-1$. 

To complete the proof we must show that $t$ is the inverse of $d$ modulo $2^n-1$. From the above observation,
we have
\begin{eqnarray*}
 t \cdot d  & = & (2^n-2^r) \cdot \frac{d  \cdot \de(r) -1}{2^r-1} + (d  \cdot \de(r) -1) +1 \\ 
 & = & (2^n-1) \cdot \frac{d  \cdot \de(r) -1}{2^r-1}  - (2^r-1)\cdot \frac{d  \cdot \de(r) -1}{2^r-1}  + (d  \cdot \de(r) -1) +1 \\
& \equiv & 1 \pmod{2^n-1}.
\end{eqnarray*}

\vspace*{0.4cm}{\bf (b)}
Set
$$\spdb(n) = \frac{d\cdot \de(n) -1}{2^n - 1}.$$
Then Corollary \ref{cor:sum} shows that
\begin{equation}\label{eq:sum}
\spdb(r) + \spdb(\ordb -r) = d+1.
\end{equation}
Multiplying (\ref{eq:main2}) by $d$, we have
\begin{align*}
d \cdot  &\left( \de(r) \sum_{i=0}^{m} 2^{\ordb \cdot i} + \left(2^{\ordb -r} - 1 - \de(\ordb -r)\right) 2^r\cdot \sum_{i=0}^{m -1} 2^{\ordb\cdot i} \right) \\
&= \left( \spdb(r)(2^r-1) + 1 \right)\sum_{i=0}^{m} 2^{\ordb\cdot i}\\
&\hspace{2cm} + 2^r\cdot d\left( 2^{\ordb-r}   - 1) - 2^r\left( \spdb(\ordb -r)(2^{\ordb -r}-1) + 1\right) \right) \sum_{i=0}^{m-1} 2^{\ordb\cdot i} \\
&=  \spdb(r)(2^r-1) \sum_{i=0}^{m} 2^{\ordb\cdot i}  + \sum_{i=0}^{m} 2^{\ordb\cdot i}\\
&\hspace{2cm} + \big( 2^r(2^{\ordb-r}-1) \left( d - \spdb(\ordb -r)\right)  - 2^r \big) \sum_{i=0}^{m-1} 2^{\ordb\cdot i} \\
&=  \spdb(r)\left( (2^n-1) +( 2^r - 2^{\ordb}) \cdot\sum_{i=0}^{m-1} 2^{\ordb \cdot i}\right)  + 1 + 2^{\ordb}\sum_{i=0}^{m-1} 2^{\ordb \cdot i} \\
&\hspace{2cm} + \left( 2^r(2^{\ordb-r}-1) \left( d - \spdb(\ordb -r)\right)  - 2^r \right) \sum_{i=0}^{m-1} 2^{\ordb\cdot i} 
\end{align*}
\begin{align*}
&=  \spdb(r)(2^n-1) + 1 + 2^{\ordb}\sum_{i=0}^{m-1} 2^{\ordb\cdot i} \\
&\hspace{2cm}+ \left( 2^r(2^{\ordb-r}-1) \left( d - \spdb(\ordb -r) - \spdb(r) \right)  - 2^r \right) \sum_{i=0}^{m-1} 2^{\ordb\cdot i} \\
&= 1 +  \spdb(r)(2^n-1) + 2^r(2^{\ordb-r}-1) \left( d - \spdb(\ordb -r) - \spdb(r) + 1\right) \sum_{i=0}^{m-1} 2^{\ordb\cdot i}\\
&\equiv 1 \pmod{2^n-1},
\end{align*}
where we apply (\ref{eq:sum}) to get the congruence modulo $2^n-1$.

\vspace*{0.4cm}{\bf (c)} This identity follows from (\ref{inverse-r}) in Proposition \ref{prop:inverse-r} and (\ref{eq:main2}) of part (b) of this theorem.
\end{proof}

Formulas (\ref{eq:main1}) and (\ref{eq:main2}) of Theorem \ref{main} imply that the binary representation of inverses modulo $2^n-1$ have a nice combinatorial structure:

\begin{cor}\label{main-comb}
Let $n>1$ with $\gcd(d, 2^n-1) =1$ and
$1 \leq r \leq \ordb -1$ such that $n \equiv r \pmod \ordb$. Set 
\begin{itemize}
\item[-] $\boldsymbol{t_n}$ be
the binary sequence of length $n$ representing $\de(n)$;
\item[-] $\boldsymbol{u_{\ordb}}$ be
the binary sequence of length $\ordb$  representing $\left( \frac{d \cdot \de(r)-1}{2^r-1} - 1 \right) \cdot\frac{2^{^{\ordb}}-1}{d} $;
\item[-] $\boldsymbol{a_r}$ be
the binary sequence of length $r$  representing $ \de (r)$;
\item[-] $\boldsymbol{\bar{b}_{\ordb -r}}$ be
the complementary sequence of the binary sequence of length $\ordb -r$ representing the  $\de (\ordb -r)$. 
\end{itemize}
Then $\boldsymbol{t_n}$ is obtained by concatenating sequences  $\boldsymbol{a_r}$,  $\boldsymbol{\bar{b}_{\ordb -r}}$ and $\boldsymbol{u_{\ordb}}$ as
follows:
$$
\boldsymbol{t_n} =  \boldsymbol{a_r}\,|\, \boldsymbol{u_{\ordb}} \,|\, \boldsymbol{u_{\ordb}}\,| \ldots |\,\boldsymbol{u_{\ordb}} 
= \boldsymbol{a_r}\,|\, \boldsymbol{\bar{b}_{\ordb -r}}\,| \, \boldsymbol{a_r}\,| \ldots |\,  \boldsymbol{\bar{b}_{\ordb -r}} \,|\, \boldsymbol{a_r}.
$$
 
\end{cor}

\begin{exmp}
Let $d =7$. Then $\ordb[7] = 3$ and for every $n$ not divisible by 3, we have $r =  1,\, 2 \equiv n \pmod{3}$. Moreover, $\de[7](1) = \de[7](2) = 1$, thus from Theorem \ref{main}, we deduce
$$\de[7](n) = 2^{n-r} + \left( \frac{6}{2^r-1} - 1 \right) \cdot \frac{2^{n-r}-1}{7}.$$
Suppose that $r=1$, then the binary representation of $\de[7](n)$, say $t$, is as follow:
$$ t = 1 \, | \, u \, | \, u \, | \, \dots \, | \, u  = 1\, | \, 10 \, | \, 1 \, | \, 10\, | \, 1 \, | \dots \, | \, 10 \, | \, 1,$$
where $u$ is a binary sequence of length 3 representing the integer $\left( \frac{d \cdot \de(r)-1}{2^r-1} - 1 \right) \cdot\frac{2^{^{\ordb}}-1}{d} = 5$.
Suppose that $r=2$, then the binary representation of $\de[7](n)$, say $t'$, is as follow:
$$ t' = 01 \, | \, u' \, | \, u' \, | \, \dots \, | \, u'  = 01 \, | \, 0 \, | \, 01 \, | \, 0 \, | \, \dots \, | \, 01,$$
where $u'$ is a binary sequence of length 3 representing the integer $\left( \frac{d \cdot \de(r)-1}{2^r-1} - 1 \right) \cdot\frac{2^{^{\ordb}}-1}{d} = 1$.

\end{exmp}

\begin{lemma}\label{lem:oeven}
Let $\ordb$ be even and let $d$ divide $2^{\ordb / 2} + 1$. 
Further, suppose $n \geq 1$ is such that $d$ and $2^n-1$ are coprime, and
let $1 \leq r \leq \ordb -1$ be the least positive residue of $n  \pmod \ordb$.
Then the following properties hold:
\begin{description}
\item[(a)]
$$wt(\de(n) ) = wt(\de(r)) + \frac{n-r}{2}.$$
\item[(b)]
\begin{equation}
wt(\de(\ordb-r)) = wt(\de(r)) + \frac{\ordb}{2}-r.
\end{equation}\label{eq:weights}
\end{description}
\end{lemma}
\begin{proof}
By Corollary \ref{main-comb} the binary weight of $\de(n)$ is
$$
wt(\de(n) ) = wt(\de(r)) + \frac{n-r}{\ordb}wt({\boldsymbol{u_{\ordb}}}),
$$
where  $\boldsymbol{u_{\ordb}}$ is
the binary sequence of length $\ordb$  representing the integer
$$
\left( \frac{d \cdot \de(r)-1}{2^r-1} - 1 \right) \cdot\frac{2^{^{\ordb}}-1}{d} = \left( \frac{d \cdot \de(r)-1}{2^r-1} - 1 \right) \cdot\frac{2^{^{\ordb /2}}+ 1}{d} \cdot (2^{^{{\ordb}/2}}-1).$$
By Lemma \ref{lem:weight}, the weight of ${\boldsymbol{u_{\ordb}}}$ is $\ordb /2$, which completes the proof of (a). The statement of (b) follows from the fact
that 
$$
{\boldsymbol{u_{\ordb}}} =  \boldsymbol{\bar{b}_{\ordb -r}} \,|\, \boldsymbol{a_r}\,,
$$
and therefore
$$
\frac{\ordb}{2} = wt({\boldsymbol{u_{\ordb}}}) = \ordb - r - wt(\de(\ordb -r)) + wt(\de(r)).
$$
\end{proof}

\begin{remark}
Let \texttt{Inverse} be an algorithm for inversion modulo $2^n-1$. 
Theorem \ref{main} and discussions of this subsection  show that for several classes of integers
finding their inverses modulo $2^n-1$ can be reduced to computations modulo $2^{n'}-1$ with $n'$
much smaller than $n$. We summarize this observation in the following algorithm.
\end{remark}

\vspace*{0.5cm}

\begin{algorithm}[H]
\KwData{positive integers $d$ and $n$ such that $\gcd(d,2^n-1)=1$.}
\KwResult{$\de(n)$, the inverse of $d$ modulo $2^n-1$.}
\If{$n=1$ or $d=1$}{
\Return 1}
\If{$d$ is even}{
\Return $\displaystyle 2^{n-1} \de[\frac{d}{2}](n) \pmod{2^n-1}$
}
$\ordb \gets$ the order of 2 modulo $d$\;
$r \gets n \pmod{\ordb}$\;\label{ordb}
\eIf{$r \neq n$}{
$A \gets \de(r)$\;
\Return $\displaystyle A\cdot 2^{n-r} + \left( \frac{d\cdot A -1}{2^r-1} -1\right) \cdot \frac{2^{n-r}-1}{d}$
\;\label{if:n} 
}{
$d' \gets d \pmod{2^n-1}$\;\label{mod:d}
\eIf{$d' \neq d$}{
\Return $\de[d'](n)$\label{if:d}
}{
\eIf{$n> \ordb/2$}{
\Return $\displaystyle \frac{\left( d + 1 - \frac{\de(\ordb-n)\cdot d -1}{2^{\ordb-n}-1} \right) (2^n-1) + 1}{d} = \de(n)$ \;\label{if:nn}
}{
Compute $\de(n)$ using \texttt{Inverse}\;\label{pb}
}
}
}\caption{Recursive inversion}
\label{algo}
\end{algorithm}

Algorithm \ref{algo} reduces the computation of the
inverse of $d$ modulo $2^n-1$ to one of modulo $2^r-1$ with $1 \leq r \leq \ordb$. In particular, this algorithm 
is effective  for  integers $d$ with $\ordb$  much smaller than $n$ or for integers $d$ with known small factors.
Furthermore, this algorithm performs good for several special integers $d$, like Gold's and Kasami's exponents considered in the next subsections.

The next two examples compute inverses using Algorithm \ref{algo}. 

\begin{exmp}
Let $n=97$ and $d=2^{11}+1= 2049$. We compute the inverse of $2049$ modulo $2^{97}-1$:
\begin{description}
\item [1:] $\ordb[2049] \leftarrow 22$
\item [2:] $r \leftarrow 9 \equiv 97 \pmod{22}$
\item [3:] $9 \neq 97$ then return $\displaystyle \de[2049](9) \cdot 2^{88} + \left( \frac{2049 \cdot \de[2049](9) -1}{511} -1 \right) \cdot \frac{2^{88} -1}{2049}$
\begin{description}
\item [3.1:] now $n=9$ and $d=2049$
\item [3.2:] $\ordb[2049] \leftarrow 22$
\item [3.3:] $r \leftarrow 9 \pmod{22}$ then 
\item [3.4:] $d' \leftarrow 5 \equiv 2049 \pmod{2^9-1}$ 
\item [3.5:] return $\de[5](9)$
\begin{description}
\item [3.5.1:] now $n=9$ and $d=5$
\item [3.5.2:] $\ordb[5] \leftarrow 4$
\item [3.5.3:] $r \leftarrow 1 \equiv 9 \pmod{4}$
\item [3.5.4:] $1 \neq 9$ then return $\displaystyle 2^8 + 3\cdot \frac{2^8 - 1}{5}$

\end{description}
\end{description}
\item [4:] $\displaystyle \de[2049](97) = \left( 2^8 + 3\cdot \frac{2^8 - 1}{5}  \right)\cdot 2^{88} + \left( \frac{2049 \cdot \left(2^8 + 3\cdot \frac{2^8 - 1}{5} \right) -1}{511} - 1  \right) \cdot \frac{2^{88} -1}{2049}$

\end{description}
Note that in this example, we do not need to call the algorithm \texttt{Inverse}.
\end{exmp}

\begin{exmp}
Let $n=101$ and $d=13$. We compute the inverse of $13$ modulo $2^{101}-1$:
\begin{description}
\item [1:] $\ordb[13] \leftarrow 12$
\item [2:] $r \leftarrow 5 \equiv 101 \pmod{12}$
\item [3:] Using \texttt{Inverse} we compute that the inverse of $13$ modulo $2^5-1$ is $12$
\item [4:] $\displaystyle \de[13](101) = 12\cdot 2^{96} + \left( \frac{155}{31} -1 \right)\cdot \frac{2^{96}-1}{13}$
\end{description}
Note that the computation of the inverse of 13 modulo $2^{101}-1$ was reduced to the one modulo $2^5-1$.
\end{exmp}

\subsection{Gold's exponent}

An integer $1 \leq d \leq 2^n-2$ is called Gold's exponent if $d=2^k+1$ and $\gcd(n,k) =1$.
The assumption $\gcd(n,k)=1$ is  necessary and sufficient for the mapping $x \mapsto x^d$ to be APN on $\Fn$.
In this section we use the term Gold's exponent to refer to the integers $d=2^k+1$ with $n/\gcd(n,k)$ odd. The assumption
$n/\gcd(n,k)$ odd ensures that $2^k+1$ is invertible modulo $2^n-1$, cf.  \cite[Lemma 11.1]{Mc}:

\begin{lemma}\label{lem:McEliece}
Let $n$ and $k$ be positive integers. Then  $\gcd(2^k + 1, 2^n - 1) = 1$ if and only if $n/\gcd(n,k)$ is odd.
\end{lemma}

The inverses of APN Gold's exponents were considered in \cite[Proposition 5]{N93}:

\begin{prop}\label{prop:nyberg}
Let $n$ be odd, and $\gcd(n,k) = 1$. Then
$$ \dego(n) \equiv \frac{2^{k(n+1)} - 1 }{2^{2k}-1} \equiv \sum_{j=0}^{\frac{n-1}{2} } 2^{2jk } \pmod{2^n - 1} $$
and
$\displaystyle{ wt(\dego(n)) = \frac{n+1}{2}}.$
\end{prop}

Note that the integer $\frac{2^{k(n+1)} - 1 }{2^{2k}-1}$ is equal to the least positive residue of the inverse
of $2^k+1$ modulo $2^n-1$ if and only if $k=1$. For $k=1$, the statement of Proposition \ref{prop:nyberg} reduces to
$\normalfont\texttt{Inv}_{3}(n) = \frac{2^{n+1}-1}{3}$ for all $n$ odd.

\begin{lemma}\label{order:gold}
Let $k \geq 1$ be an integer. The order of $2$ modulo $2^k+1$ is $\theta_{2^k+1} =2k$. 
\end{lemma}
\begin{proof}
Clearly, $k$ is the smallest positive integer satisfying
$$
 2^{k} = -1 \pmod {2^k +1},
$$
implying that $\theta_{2^k+1} = 2k$.
\end{proof}

Lemma \ref{order:gold} and Theorem \ref{main} show that to invert a fixed Gold's exponent $2^k+1$ modulo all $2^n-1~~( n \geq 1)$ it is enough
to obtain the inverses $\dego(r)$ for $1 \leq r < 2k$. In some of the arguments of this section it will be enough to consider only $1 \leq r \leq k$, because of the 
following easy observation:

\begin{claim}\label{claim:gold}
Let $0\leq r <k$. Then $2^k+1 \equiv 2^k\cdot(2^r+1) \pmod {2^{k+r}-1}$. In particular, the least positive residues
of inverses of $2^r+1$ and $2^k+1$ modulo $2^{k+r}-1$ have the same binary weight.
\end{claim}

The following theorem summarizes the main results on the inverses of Gold exponents.

\begin{thm}\label{main:gold}
Let $n, k \geq 1$ with $\gcd(n, k) =s$ and $n/s$ odd. Set $r$ be the least positive residue of $n$ modulo $2k$.
 Then
\begin{description}
\item[(a)]
$$
\dego(n) = \dego(r) \cdot 2^{n-r} + \left( \frac{(2^k+1) \cdot \dego(r)-1}{2^r-1} - 1 \right) (2^k-1) \frac{2^{n-r}-1}{2^{2k}-1}.
$$
\item[(b)]
 Set 
\begin{itemize}
\item[-] $\boldsymbol{g_n}$ be
the binary sequence of length $n$ representing $\dego(n)$;
\item[-] $\boldsymbol{w_{k}}$ be
the binary sequence of length $k$  representing $\frac{(2^k+1) \cdot \dego(r)-1}{2^r-1} - 2  $;
\item[-] $\boldsymbol{a_r}$ be
the binary sequence of length $r$  representing $ \dego (r)$.
\end{itemize}
Then $\boldsymbol{g_n}$ is obtained by concatenating sequences  $\boldsymbol{a_r}$,  $\boldsymbol{w_k}$ and $\boldsymbol{\bar{w}_k}$ as
follows:
$$
\boldsymbol{g_n} 
= \boldsymbol{a_r}\,|\,\boldsymbol{w_k}\,|\,  \boldsymbol{\bar{w}_{k}} \,|\, \ldots |\, \boldsymbol{w_k} \,|\,  \boldsymbol{\bar{w}_{k}}. 
$$
In particular, $wt(\boldsymbol{g_n}) = wt(\boldsymbol{a_r}) + \frac{n-r}{2}$. 
\item[(c)]
The binary weight of $\dego(n)$ is $$ \frac{n-s}{2} +1 = \frac{n-s+2}{2}.$$
\end{description}
\end{thm}

\begin{proof} Statement (a) follows from Theorem \ref{main} (a). To prove (b), set
$$
u := \frac{(2^k+1) \cdot \dego(r)-1}{2^r-1} - 1.
$$
From the definition of $\dego(r)$, it easily follows that $u < 2^k$. Lemma \ref{lem:weight} combined with the part (a) of this theorem complete the proof of (b).
We prove statement (c) by induction on $k_1$, where $k_1 = k/s$. If $k_1 =1$ (or equivalently $k =s$), then  $n \equiv s \pmod {2s}$, since $n/s$ is odd.
Consequently, $r =s$. Observe that $\normalfont\texttt{Inv}_{2^s+1}(s) = 2^{s-1}$. From (b), we have 
$$
wt\left(\normalfont\texttt{Inv}_{2^s+1}(n) \right)= wt\left( \normalfont\texttt{Inv}_{2^s+1}(s) \right) + \frac{n-s}{2} = 1 + \frac{n-s}{2}.
$$
Suppose now that the statement holds for all $k_1 < \ell$ and take $k = s \ell$. Let $1 \leq r < 2s \ell$ be the residue of $n \pmod {2s \ell}$.
Then by (b) the problem reduces to finding the weight of the inverse of $2^{s\ell}+1$ modulo $2^r-1$. Note that if $r < s\ell$, then there is $\ell' < \ell$ such that
$$
2^{s\ell}+1 \equiv 2^{s\ell'} +1 \pmod {2^r-1}.
$$
Hence the inverse of $2^{s\ell}+1$ modulo $2^r-1$ is equal to the one of $2^{s\ell'} +1$, and we get
\begin{eqnarray*}
wt\left(\normalfont\texttt{Inv}_{2^{s\ell}+1}(n) \right) & =& wt\left( \normalfont\texttt{Inv}_{2^{s\ell}+1}(r) \right) + \frac{n-r}{2} \\
 & =& wt\left( \normalfont\texttt{Inv}_{2^{s\ell'}+1}(r) \right) + \frac{n-r}{2}\\
 & = & 1 + \frac{r-s}{2} +  \frac{n-r}{2}\\
& = & 1 + \frac{n-s}{2}.
\end{eqnarray*}
To complete the proof we must prove the statement for the case $r \geq s\ell$. Let $r = s\ell + r'$. Then $r' = s \ell'$ for some $\ell ' < \ell$.
Using Claim \ref{claim:gold}, the binary weight of the inverses of  $2^{s\ell}+1$ and $2^{s\ell'}+1$ modulo $2^{s\ell + s\ell'}-1$ are equal, which completes the proof.
\end{proof}

\begin{remark}
The reason, why it was possible to determine the binary weight of the inverse of a Gold exponent in (c) of Theorem \ref{main:gold},
is the fact that Algorithm \ref{algo} does not call the algorithm \texttt{Inverse}, when computing the inverse for a Gold exponent.
\end{remark}

\subsection{Kasami's exponent}

We call integers $d = 2^{2k} - 2^k + 1$, where $k$ is a positive integer, Kasami exponents. Such exponents define APN mappings on $\Fn$
if and only if $\gcd(k,n)=1$. The next lemma summarize properties of Kasami exponents, which we need for later results:

\begin{lemma}\label{thm:kasaperm}
Let $k, n$ be positive integers. Then
$\gcd(2^{2k}-2^k + 1, 2^n-1)=1$ if and only if one of the following cases occurs:
\begin{itemize}
\item[(a)] $\frac{n}{\gcd(n,k)}$ is odd, that is $\gcd(n,2k) = \gcd(n,k)$;
\item[(b)] $\frac{n}{\gcd(n,k)}$ is even,  $k$ is even 
 and $\gcd(k,n) = \gcd(3k,n)$.
\end{itemize}
Equivalently, $\gcd(2^{2k}-2^k + 1, 2^n-1)=1$ if and only if one of the following cases occurs:
\begin{itemize}
\item $n$ is odd and $k\geq 1$ is arbitrary 
\item $ n=2^ra~\textrm{and}~k=2^rb,~\textrm{where $a$ is odd and  $1\leq r$}. $
\item $ n=2^r3^ua~\textrm{and}~k=2^s3^vb,$
where $b$ is odd, $1\leq s <r$ and $0\leq u\leq v$.
\end{itemize}
\end{lemma} 
%%%%%%%%%%%%%%%%%%%%%%
To prove this lemma above, we need to recall some propositions.
\begin{prop}\label{lem:kasa_quad}
Let $k$ be an integer. Then
\begin{displaymath}
\gcd(2^{2k} - 2^k + 1, 2^k + 1) = \left\{ \begin{array}{ll}
1 & \textrm{ if $k$ is even,}\\
3 & \textrm{ otherwise.}
\end{array} \right.
\end{displaymath}
\end{prop}
\begin{proof}
\begin{align*}
\gcd(2^{2k} - 2^k + 1, 2^k + 1) &= \gcd(2^{2k} - 2^{k+1}, 2^k + 1) \\
&= \gcd(2^{k+1}(2^{k-1} - 1), 2^k + 1) \\
&= \gcd(2^{k-1} - 1, 2^k + 1) \\
&= \left\{ \begin{array}{ll}
1 & \textrm{ if $k$ is even,}\\
2^{\gcd(k,k-1)} + 1 = 3 & \textrm{ otherwise.}
\end{array} \right.
\end{align*}
Here we apply lemma \ref{lem:McEliece}.
\end{proof}

\begin{prop}\label{lem:keven}
For any even integer $n$ and $2k<n$, if
$$\gcd(2^{2k} - 2^k + 1, 2^n - 1) = 1$$
then $k$ is even.
\end {prop}
\begin{proof}
Suppose that $k$ is odd.
So, from Proposition \ref{lem:kasa_quad}, we have that
$$\gcd(2^{2k} - 2^k + 1, 2^k + 1) =3$$
But $3$ divides $\gcd(2^n - 1, 2^k + 1)$ because $n$ is even.\\
It implies that $3$ divides $\gcd(2^{2k} -2^k + 1, 2^n - 1)$, a contradiction.
\end{proof}

Now we can prove Lemma \ref{thm:kasaperm}.
\begin{proof}[Proof of Lemma \ref{thm:kasaperm}]
Let $d=2^{2k} - 2^k + 1$. We want to determine when $\gcd(d, 2^n-1)=1$.
The first condition {\bf (i)} means that
$$\gcd(2^k + 1, 2^n - 1) = 1$$ (from (\ref{lem:McEliece})).
Then we deduce that  $\gcd(2^{rk} + 1, 2^n - 1) = 1$ for all odd $r$.
In particular, this holds for $r=3$. Using 
$$(2^k + 1)(2^{2k} - 2^k + 1) = 2^{3k} + 1,$$
we conclude that in this case $\gcd(d, 2^n-1)=1$.
Note that  {\bf (i)} is satisfied for any odd $n$.

Now, we assume that $n/gcd(n,k)$ is even; so $n$ is even.
 From Proposition \ref{lem:keven}, we know that $\gcd(d, 2^n-1)\ne 1$ for odd
$k$. So  $k$ must be even.

Again from (\ref{lem:McEliece}), we have in this case
$$\gcd(2^k + 1, 2^n - 1) = 2^{\gcd(k,n)}+1=e,~e> 3,$$
and similarly
\[
\gcd(2^{3k} + 1, 2^n - 1) = 2^{\gcd(3k,n)}+1=u,
\]
where $e$ divides $u$ since $2^{3k} + 1=(2^k + 1)d$.
But, from Proposition \ref{lem:kasa_quad}, we have
$\gcd(d, 2^k + 1) = 1$ since $k$ is even.
Hence $\gcd(d, 2^n-1)=1$ if and only if $u=e$ or, equivalently,
$\gcd(3k,n)=\gcd(k,n)$, completing the condition {\bf (ii)}.
\end{proof}

%%%%%%%%%%%%%%%%%%%%%%%%%

\begin{prop}
Let $k>1$ be an integer. Then the order of 2 modulo the Kasami exponent $2^{2k} - 2^k + 1$ is $\ordb = 6k$.
\end{prop}
\begin{proof}
It is enough to show that $3k$ is the least positive integer satisfying $2^{3k} \equiv -1  \pmod{2^{2k} -2^k + 1}$.
The congruence $2^{3k} \equiv -1  \pmod{2^{2k} -2^k + 1}$ holds, since clearly $2^{2k} -2^k + 1$ divides $2^{3k}+1$.
Let  an integer $0 < \sigma <3k$ be such that $2^{\sigma} \equiv -1 \pmod{2^{2k} - 2^k + 1}$. Then $\sigma \geq 2k-1$, since 
$1 < 2^l < 2^{2k} -2^k $ for $ 0 < l \leq 2k-1$. Hence suppose $ \sigma = 2k + i$ with $ 0 \leq i < k$. Note that
$$
2^{2k+i} \equiv 2^i (2^k-1) \pmod{2^{2k} -2^k + 1}.
$$
To complete the proof it remains to observe that $2^i (2^k-1) <2^{2k} -2^k$ if $ 0 \leq i < k$. 

\end{proof}

Theorem \ref{main} implies for the Kasami exponents:

\begin{thm}\label{thm:kasami}
Let $n, k \geq 1$ with $\gcd(2^{2k}-2^k+1,2^n-1) = 1$. Let $r$ be the least positive residue of $n$ modulo $6k$.
 Then
\begin{description}
\item[(a)] $\dekas(n)$ is equal to
\begin{eqnarray*}
 \dekas(r) \cdot 2^{n-r} + \left( \frac{(2^{2k}-2^k+1) \cdot \dekas(r)-1}{2^r-1} - 1 \right)  \frac{2^{n-r}-1}{2^{2k}-2^k+1}.
\end{eqnarray*}
\item[(b)]
Let
\begin{itemize}
\item[-] $\boldsymbol{g_n}$ be
the binary sequence of length $n$ representing $\de[2^{2k}-2^k+1](n)$;
\item[-] $\boldsymbol{w_{3k}}$ be
the binary sequence of length $3k$  representing
$$
\left( \frac{(2^{2k}-2^k+1) \cdot \dekas(r)-1}{2^r-1} - 1 \right)\cdot (2^k+1) -1;
$$
\item[-] $\boldsymbol{a_r}$ be
the binary sequence of length $r$  representing $ \de[2^{2k}- 2^k+1] (r)$.
\end{itemize}
Then $\boldsymbol{g_n}$ is obtained by concatenating sequences  $\boldsymbol{a_r}$,  $\boldsymbol{w_{3k}}$ and $\boldsymbol{\bar{w}_{3k}}$ as
follows:
$$
\boldsymbol{g_n} 
= \boldsymbol{a_r}|\boldsymbol{w_{3k}}|  \boldsymbol{\bar{w}_{3k}} | \ldots | \boldsymbol{w_{3k}} |  \boldsymbol{\bar{w}_{3k}}.
$$
In particular, $wt(\boldsymbol{g_n}) = wt(\boldsymbol{a_r}) + \frac{n-r}{2}$. 
\end{description}
\end{thm}
\begin{proof}
The statement follows from Theorem \ref{main} and Lemma \ref{lem:weight}, similarly to the proof of Theorem \ref{main:gold}.
\end{proof}

\begin{exmp}
Consider the Kasami exponent $13 = 2^4-2^2+1$.  Lemma \ref{thm:kasaperm} shows that $\gcd(13, 2^n-1)=1$ if and only if $n \not\equiv 0 \pmod{12}$. 
Then using Theorem \ref{thm:kasami}, we get:
\begin{itemize}
\item if $n \equiv 1 \pmod{12}$ then
$$
\de[13](n) = 2^{n-1} + 11 \cdot\frac{2^{n-1}-1}{13},
$$
since $\de[13](1) = 1$. The  binary weight of $\de[13](n)$ is $(n+1)/2$.
\item if $n \equiv 2 \pmod{12}$ then
$$
\de[13](n) = 2^{n-2} + 3 \cdot\frac{2^{n-2}-1}{13}
$$
since $\de[13](2) = 1$. The  binary weight of $\de[13](n)$ is $n/2$.

\item if $n \equiv 3 \pmod{12}$ then
$$
\de[13](n) = 6\cdot 2^{n-3} + 10\cdot\frac{2^{n-3}-1}{13}
$$
since $\de[13](3) = 6$. The  binary weight of $\de[13](n)$ is $(n+1)/2$.
\item if $n \equiv 4 \pmod{12}$ then
$$
\de[13](n) = 7\cdot 2^{n-4} + 5\cdot \frac{2^{n-4}-1}{13}
$$
since $\de[13](4) = 7$. The  binary weight of $\de[13](n)$ is $(n+2)/2$.

\item if $n \equiv 5 \pmod{12}$ then
$$
\de[13](n) = 12 \cdot 2^{n-5} +  5\cdot\frac{2^{n-5}-1}{13}
$$
since $\de[13](5) = 12$. The  binary weight of $\de[13](n)$ is $(n-1)/2$.
\item if $n \equiv 6 \pmod{12}$ then
$$
\de[13](n) = 34\cdot 2^{n-4} + 6\cdot \frac{2^{n-4}-1}{13}
$$
since $\de[13](6) = 34$. The  binary weight of $\de[13](n)$ is $(n-2)/2$.\\

Then using identity (\ref{eq:weights}), we obtain that:

\item if $n \equiv 7 \pmod{12}$ then the  binary weight of $\de[13](n)$ is $(n-1) /2$,
since $wt(\de[13](7)) = 3.$

\item if $n \equiv 8 \pmod{12}$ then the  binary weight of $\de[13](n)$ is $(n+2) /2$,
since $wt(\de[13](8)) = 5.$

\item if $n \equiv 9 \pmod{12}$ then
the  binary weight of $\de[13](n)$ is $(n+1)/2$, since  $wt(\de[13](9)) = 5.$

\item if $n \equiv 10 \pmod{12}$ then the  binary weight of $\de[13](n)$ is $n /2$,
since $wt(\de[13](10)) = 5.$

\item if $n \equiv 11 \pmod{12}$ then
the  binary weight of $\de[13](n)$ is $(n+1)/2$, since  $wt(\de[13](11)) = 6.$

\end{itemize}
\end{exmp}

{\bf Open question:} Is it possible to express the binary weight of the inverse of Kasami's exponent $2^{2k}-2^k+1$ modulo $2^n-1$ in terms of $k$ and $n$?\\

\begin{table}[htb]
\begin{center}
\begin{tabular}{|c||c|c|c|c|c|c|c|c|c|}
\hline
$r$ &1&3&5&7&9&11&13&15&17\\
\hline
$wt(\de(r))$ &1&1&2&4&2&6&6&7&9\\  
\hline
\end{tabular} 
\caption{Weights of the inverse of $d = 2^6 - 2^3 +1$ modulo $2^r-1$, $1 \leq r \leq 17$}
\label{tab:k=3}
\end{center} 
\end{table}

\begin{table}[htb]
\begin{center}
\begin{tabular}{|c||c|c|c|c|c|c|c|c|c|c|c|c|}
\hline
$r$ &1&2&3&4&5&6&7&8&9&10&11&12\\
\hline
$wt(\de(r))$ &1&1&2&1&3&2&4&5&5&3&4&2\\
\hline
\hline
$r$ &13&14&15&16&17&18&19&20&21&22&23& \\
\hline
$wt(\de(r))$ &5&5&8&9&9&8&10&9&11&11&12& \\
\hline
\end{tabular} 
\caption{Weights of the inverse of $d=2^8 - 2^4 +1$ modulo $2^r-1$, $1 \leq r \leq 23$}
\label{tab:k=4}
\end{center} 
\end{table}

\begin{table}[htb]
\begin{center}
\begin{tabular}{|c||c|c|c|c|c|c|c|c|c|c|c|c|c|c|c|}
\hline
$r$ &1&3&5&7&9&11&13&15&17&19&21&23&25&27&29 \\
\hline
$wt(\de(r))$ &1&2&1&3&5&5&7&2&9&9&11&11&11&14&15\\
\hline
\end{tabular} 
\caption{Weights of the inverses of $d=2^{10} - 2^5 +1$ modulo $2^r-1$, $1 \leq r \leq 29$}
\label{tab:k=5}
\end{center} 
\end{table}

Tables \ref{tab:k=3}--\ref{tab:k=5} contain the weights of the inverses for the Kasami exponents defined with $k=3,4,5$.
Some of the values in these tables follow from the results of Proposition \ref{prop:kas-weights}:

\begin{prop}\label{prop:kas-weights}
Let $k, n\geq 1$ and $\gcd(2^{2k}-2^k+1, 2^n-1)=1$. 
\begin{itemize}

\item[(a)] If $n \equiv b \pmod{6k}$ with $b\geq 1$ a divisor of $k$, then $\dekas(k/b) = 1$ and
$$
\dekas(n) = 2^{n-k/b} + \left( \frac{2^k(2^k-1)}{2^{k/b}-1} -1 \right) \cdot \frac{2^{n-k/b}-1}{2^{2k}-2^k+1}.
$$

\item[(b)] $\dekas(k-1)  =  \frac{2^k-1}{3}$.

\item[(c)]
 $\dekas(k+1)  = \frac{2^{k+2}-1}{3}+1$.

\item[(d)]
$\dekas(2k) = (2^{k}+2)\cdot\frac{2^k-1}{3} +1$.

\item[(e)] If $n \equiv 3k/b \pmod{6k}$, where $b\geq 1$ is a divisor of $k$ with $\gcd(b,3)=1$. Let $b'$ be the least positive residue of $b$ modulo 3, then $\dekas(3k/b) = 2^{3k/b-1} + 2^{b'\cdot k/b - 1}$.

\item[(f)]
$\dekas(4k) = (1 + 2^{k+1} + 2^{2k} + 2^{3k+1})\cdot\frac{2^k-1}{3} +2^k$.

\item[(g)]
 $\dekas(5k) = 2^{5k}-2^{4k} + 2^{2k} + 2^k -1 \equiv 2^{2k}(2^{4k} - 2^{2k} + 1) \pmod{2^{5k}-1}$.

\item[(h)]
$\dekas(6k - 1) = (2^{3k}-1)(2^k+1)$.
\end{itemize}

\end{prop}

Note that in cases (c)--(k) of the above proposition $n$ depends on $k$, and hence we are in a similar situation discussed in Section \ref{nkdepend}.

For any $b$ dividing $k$ and being coprime to 5, we conjecture that $\dekas(5k/b)$ is congruent to a Kasami exponent, that is, $\dekas(5k/b) \equiv 2^u (2^{2v} - 2^v +1) \pmod{2^{5k/b}-1}$ for some integers $u$ and $v$.

Finally we want to remark that in some cases the computation of inverses for Kasami exponents can be reduced to the one for the Gold exponents.

\begin{claim}
Let the positive integers $k,\,n$ be such that $\frac{n}{\gcd(n,k)}$ is odd. Then both $2^k+1$ and  $2^{3k}+1$ are coprime to $2^n-1$, and therefore
$$\de(n) \equiv (2^k+1)\de[2^{3k}+1](n) \pmod{2^n-1}$$
holds. 
\end{claim}

\subsection{$2^k-1$ exponent}
In \cite{BCC11} it is shown that the power mappings with exponents $2^k - 1$ have interesting properties for cryptological applications.
It is well known that $\gcd(2^n-1, 2^k-1) = 2^{\gcd(n,k)}-1$, and therefore  $2^k-1$ is invertible modulo $2^n-1$ if and only if $n$ and $k$ are coprime.
Moreover this indicates also that the calculation of the inverse of $2^k-1$ modulo $2^n-1$ reduces to the one of $k$ modulo $n$.

\begin{thm}\label{proposition-basic1}
Let $n, k\geq 2$ be  coprime integers. Then
$$
\normalfont{\texttt{Inv}}_{2^k-1}(n) \equiv \frac{2^{k\cdot s}-1}{2^k -1}\pmod{2^n-1},
$$
where $s$ is any positive integer satisfying $s \cdot k \equiv 1 \pmod{n}$.
More precisely, if  $k_n^{-1}$ is  the least positive residue of the inverse of $k$ modulo $n$, then
\begin{equation}\label{eqn-basic1}
\normalfont{\texttt{Inv}}_{2^k-1}(n) = \sum_{i=0}^{k_n^{-1}-1}2^{ki \!\!\!\pmod{n}}.
\end{equation}
\end{thm} 
\begin{proof}
Let $s\cdot k - 1 = n\cdot m$.
Then
$$
(2^k-1)\cdot \frac{2^{k\cdot s}-1}{2^k -1} = 2^{k\cdot s}-1 = 2^{nm+1} -1 \equiv 1 \pmod{2^n-1}.
$$
The second statement follows if we put $s = k_n^{-1}$.
\end{proof}

\section{Conclusion}

This paper is motivated by a problem to find explicitly the inverses of the known APN exponents.
We succeed this for Welch and Dobbertin exponents. The case of the exceptional APN exponents,
that is of the Gold and Kasami exponents, is  more difficult as we show in Section
\ref{section:exceptionals}. For  the Gold exponents $2^k+1$, we found the binary weights of their inverses modulo $2^n-1$ in terms of $n$ and $k$. For the Kasami exponents $2^{2k}-2^k+1$, 
we showed that the binary weight of the inverses is uniquely defined by the binary weight of its inverse modulo $2^r-1$, where $r$ is the least positive residue of $n$ modulo $\ordb[2^{2k}-2^k+1] = 6k$.
Presently, it is not clear to us whether we may expect more explicit results on Kasami exponents than those given in Theorem \ref{thm:kasami}.

Generally, for a fixed positive integer we considered the function $\de$, which maps
$n$ to the least positive residue of the inverse of $d$ modulo $2^n-1$, where $d$ is a fixed positive integer.  We are not aware whether the function $\de$ was studied before.
We think that a better understanding of $\de$ in general, as well as for special values of $d$, is a fundamental problem deserving a further development.
In particular, it would be interesting to see if there are any connections
with the algebraic feedback shift register sequences (see \cite{klapper}) yielding new insights on the problem.

\end{document}